\documentclass[12pt,a4paper]{amsart}
\usepackage[top=1.15in, bottom=1.15in, left=1.17in, right=1.17in]{geometry}
\usepackage{amssymb,amsmath}
\usepackage{pdflscape}
\usepackage{amscd}
\usepackage{fancyhdr}
\usepackage{tikz}
\usetikzlibrary{matrix,arrows,cd}
\usepackage{verbatim}
\usepackage{todonotes}

\DeclareMathOperator{\Hom}{Hom}
\DeclareMathOperator{\Ext}{Ext}
\DeclareMathOperator{\End}{End}
\newcommand{\C}{\mathbb{C}}

\newtheorem{theorem}{Theorem}
\newtheorem{corollary}[theorem]{Corollary}
\newtheorem{lemma}[theorem]{Lemma}
\newtheorem{proposition}[theorem]{Proposition}

\theoremstyle{definition}

\newtheorem{remark}[theorem]{Remark}

\title{The Mukai conjecture for Fano quiver moduli}
\author{Markus Reineke}
\address{ 
Markus Reineke\\ Ruhr-University Bochum\\ Faculty of Mathematics\\
Universit\"ats\-strasse 150\\
44780 Bochum (Germany)}
\email{Markus.Reineke@ruhr-uni-bochum.de}

\begin{document}

\begin{abstract} We verify the Mukai conjecture for Fano quiver moduli spaces associated to dimension vectors in the interior of the fundamental domain.
\end{abstract}
\maketitle
\parindent0pt

For a Fano variety $X$, that is, a smooth projective algebraic variety with ample anticanonical bundle, a famous conjecture of Mukai \cite{Muk} asserts that the inequality $$\dim X\geq r_X\cdot(i_X-1),$$ where $r_X$ denotes the Picard rank of $X$ and $i_X$ the index of $X$ (the maximal integer dividing the anticanonical class) holds, with equality only for powers of projective space.

In \cite{FRS}, the moduli spaces $M_{\bf d}^{\rm st}(Q)$ of stable quiver representations \cite{RModuli} which are Fano varieties are identified, forming a rather special  (for example, always being rational \cite{SchB} and of algebraic cohomology \cite{KW}), but arbitrarily high-dimensional, class of Fano varieties.

In this note, we verify the Mukai conjecture for Fano quiver moduli spaces (in the spirit of such a verification for toric varieties \cite{Ca06}, horospherical varieties \cite{Pas10} and symmetric varieties \cite{GH}), under a reasonable genericity assumption on the dimension vector ${\bf d}$, to be discussed below.\\[1ex]
So let $Q$ be a connected finite acyclic quiver with set of vertices $Q_0$ and arrows written $\alpha:i\rightarrow j$ for $i,j\in Q_0$, and let $$\langle{\bf d},{\bf e}\rangle=\sum_{i\in Q_0}d_ie_i-\sum_{\alpha:i\rightarrow j}d_ie_j,$$ for ${\bf d},{\bf e}\in\mathbb{Z}Q_0$, its Euler form; we also consider its symmetrization $({\bf d},{\bf e})$ as well as its antisymmetrization $\{{\bf d},{\bf e}\}$. For a dimension vector ${\bf d}=\sum_id_i{\bf i}\in\mathbb{Z}_{>0}Q_0$, the linear form $\{{\bf d},\_\}$ defines a stability in the sense of \cite{King}, for which we assume in the following that there exists a stable representation of dimension vector ${\bf d}$, that ${\bf d}$ is coprime (that is, $\{{\bf d},{\bf e}\}\not=0$ for all $0\not={\bf e}\lneqq{\bf d}$), and that ${\bf d}$ is amply stable (there are no divisors in the unstable locus of the representation space, see \cite{FRS}). The main result of \cite{FRS} then asserts:
 
 \begin{theorem}[\cite{FRS}]\label{frs}The moduli space $M_{\bf d}^{\rm st}(Q)$ parametrizing isomorphism classes of $\{{\bf d},\_\}$-stable complex representations of $Q$ of dimension vector ${\bf d}$ is a complex Fano variety of dimension $1-\langle{\bf d},{\bf d}\rangle$, Picard rank $|Q_0|-1$, and index $\gcd(\{{\bf d},{\bf i}\}:i\in Q_0)$.
 \end{theorem}
 
 That ${\bf d}$ admits a stable representation can, in general, only be verified using a highly recursive criterion \cite{Scho}, except if ${\bf d}$ belongs to the so-called fundamental domain \cite{Kac}, that is, $({\bf d},{\bf i})\leq 0$ for all $i\in Q_0$. We have to slightly strengthen this reasonable assumption to obtain our main result:
 
 \begin{theorem}\label{main} In addition to the previous assumptions, assume that $({\bf d},{\bf i})\leq -2$ for all $i\in Q_0$. Then $M_{\bf d}^{\rm st}(Q)$ fulfills the Mukai conjecture. \end{theorem}
 
 We will reduce the statement to the following elementary counting argument:

\begin{lemma} Let $a_1,\ldots,a_k,b_1,\ldots,b_l$ be positive integers. Assume that no proper partial sums coincide, that is, if $\sum_{i\in I}a_i=\sum_{j\in J}b_j$ for subsets $I\subset\{1,\ldots,k\}$ and $J\subset\{1,\ldots,l\}$, then $I=J=\emptyset$ or $I=\{1,\ldots,k\}$, $J=\{1,\ldots,l\}$. Then $$\sum_ia_i+\sum_jb_j\geq 2(k+l-1),$$ with equality only if $l=1$ and all $a_i$ equal to one, or $k=1$ and all $b_j$ equal to one.
\end{lemma}

\begin{proof} Consider the numbers $$d_{i,j}=a_1+\ldots+a_i-b_1-\ldots-b_j$$ for $0\leq i\leq k$, $0\leq j\leq l$. Their mutual differences are sums or differences of partial sums of the $a_i$ and the $b_j$, thus they are either manifestly positive or negative, or non-zero by the assumption of the lemma, except possibly for $d_{0,0}$, $d_{k,l}$. So the $d_{i,j}$ form at least $(k+1)(l+1)-1$ mutually different integers in the interval $[-\sum_jb_j,\sum_ia_i]$, and thus $$\sum_ia_i+\sum_jb_j\geq(k+1)(l+1)-2=(k-1)(l-1)+2(k+l-1)\geq 2(k+l-1).$$ Moreover, if equality holds, then $k=1$ or $l=1$ by the second inequality, and the first inequality forces $d_{0,0}=d_{k,l}$, which reads $\sum_ia_i=\sum_jb_j$. This in turn implies (assuming $l=1$ without loss of generality) $\sum_ia_i=k+l-1=k$, and thus all $a_i=1$.\end{proof}

In light of Theorem \ref{frs}, the claim of Theorem \ref{main} reduces to the following statement, since, by \cite[Section 5.3]{FRS}, for thin dimension vectors of $m$-thickened subspace quivers the moduli spaces realize powers of projective space:

\begin{proposition} Assume that  ${\bf d}$ is $\{{\bf d},\_\}$-coprime and that $({\bf d},{\bf i})\leq -2$ for all $i\in Q_0$. Then
$$1-\langle{\bf d},{\bf d}\rangle\geq(|Q_0|-1)\cdot (\gcd(\{{\bf d},{\bf i}\} : i\in Q_0)-1),$$
with equality only if $Q$ is an $m$-thickened subspace quiver or its opposite, and ${\bf d}$ is thin.
\end{proposition}

\begin{proof}  For all $i\in Q_0$, define nonnegative integers $$\alpha_i=d_i-\langle{\bf i},{\bf d}\rangle,\;\beta_i=d_i-\langle{\bf d},{\bf i}\rangle,$$ so that, by assumption, $$-2\geq({\bf d},{\bf i})=2d_i-\alpha_i-\beta_i\mbox{ and }\alpha_i-\beta_i=\{{\bf d},{\bf i}\}.$$ We can then estimate $1-\langle{\bf d},{\bf d}\rangle=$ $$=1+\sum_i((\alpha_i+\beta_i)/2-d_i)d_i\geq 1+\sum_i((\alpha_i+\beta_i)/2-1)=\sum_i(\alpha_i+\beta_i)/2-(|Q_0|-1).$$ The claim thus follows if $$\sum_i(\alpha_i+\beta_i)\geq 2(|Q_0|-1)\cdot m$$ for $m=\gcd(\{{\bf d},{\bf i}\} : i\in Q_0)$. We can thus write $$\alpha_i-\beta_i=\varepsilon_i\gamma_i\cdot m$$ for a sign $\varepsilon_i$ and nonnegative integral $\gamma_i$. The estimate $$\alpha_i+\beta_i\geq |\alpha_i-\beta_i|$$ then reduces the claim to $$\sum_i\gamma_i\geq 2(|Q_0|-1).$$ The coprimality assumption reads $$0\not=\sum_i\varepsilon_i\gamma_ie_i$$ for all $0\not={\bf e}\lneqq{\bf d}$. In particular, $\gamma_i$ is positive for all $i\in Q_0$. Denoting the $\gamma_i$ for positive $\varepsilon_i$ by 
$a_1,\ldots, a_k$, and those for negative $\varepsilon_i$ by $b_1,\ldots,b_l$, coprimality implies that the assumptions of the previous lemma are satisfied, and the first claim there yields the desired estimate. 
Now assume that equality holds in all estimates. By the lemma we thus find (without loss of generality) a unique vertex $i_0\in Q_0$ such that $$\alpha_i-\beta_i=m\mbox{ for all }i\not=i_0\mbox{, whereas }\alpha_{i_0}-\beta_{i_0}=(1-|Q_0|)\cdot m.$$ The equalities $\alpha_i+\beta_i=|\alpha_i-\beta_i|$ then imply $\beta_i=0$, and thus $\alpha_i=m$, for all $i\not=i_0$, and $\alpha_{i_0}=0$, and thus $\beta_{i_0}=(|Q_0|-1)\cdot m$. By definition of $\alpha_i,\beta_i$, we find that $Q$ is the $(m/d_{i_0})$-thickened $s$-subspace quiver, with unique sink $i_0$ and $s=|Q_0|-1$, and that
$$d_{i_0}=\sum_{i\not=i_0}d_i/s.$$ Furthermore, the first step in the chain of inequalities shows that $d_i=1$ or $d_i=m/2-1$ for all $i\not=i_0$, and that $d_{i_0}=1$ or $d_{i_0}=sm/2-1$. A quick estimate shows that the latter is impossible, thus ${\bf d}$ is thin.
\end{proof}

{\bf Acknowledgments:} The author would like to thank H.~Franzen and S.~Sabatini for helpful discussions on the topic. This work is supported by the DFG CRC-TRR 191 ``Symplectic structures in geometry, algebra and dynamics'' (281071066).

\end{document}